\documentclass[11pt]{amsart}
\usepackage{pdfsync}
\usepackage{enumerate}
\usepackage{graphicx}
\usepackage{latexsym}
\usepackage{amsfonts}
\usepackage[utf8]{inputenc}
\usepackage{amsthm}
\usepackage{amsmath}
\usepackage{amssymb}
\usepackage[
            unicode,
            breaklinks=true,
            colorlinks=true]{hyperref}
\usepackage[active]{srcltx}
\usepackage{amscd}
\usepackage[all]{xy} 
\usepackage{mathrsfs} 
\usepackage{stmaryrd} 
\usepackage{amsopn} 
\usepackage{eepic}
\usepackage{epic}
\usepackage{eucal}
\usepackage{epsfig}
\usepackage{psfrag}

\newtheorem{theorem}{Theorem}[section]

\newtheorem{lemma}[theorem]{Lemma}

\newcommand{\R}{{\mathbb R}}

\newcommand{\Z}{{\mathbb Z}}

\newcounter{todo}
\renewcommand{\geq}{\geqslant}
\renewcommand{\leq}{\leqslant}

\newcommand{\abs}[1]{\left|#1\right|}

\newenvironment{remark}{\refstepcounter{theorem}\par\medskip\noindent{\bf
Remark~\thetheorem.}}{\unskip\nobreak\hfill\hbox{ $\oslash$}\par\bigskip}

\newenvironment{question}{\refstepcounter{theorem}\par\medskip\noindent{\bf
Question~\thetheorem}}{\unskip\normalfont \unskip\nobreak\hfill\hbox{\bigskip}}

\begin{document}

\title{Sharp symplectic embeddings of cylinders}
\author{\'Alvaro   Pelayo\,\,\,\,\,\,\,\,\,\,\,\,\,\,\,\,\,\,\,\,\,\,\,\,\,\,San V\~{u}
  Ng\d{o}c}

\maketitle

\begin{abstract}
 We show that the cylinder 
${\rm Z}^{2n}(1):={\rm B}^2(1) \times  \mathbb{R}^{2(n-1)} $ embeds symplectically into 
${\rm B}^4(R)
  \times \mathbb{R}^{2(n-2)}$ if $R\geq \sqrt{3}$.
         \end{abstract}

\section{\textcolor{black}{Introduction}} \label{sec:intro} 

On $\mathbb{R}^{2n}$ with points $(x_1,y_1,\ldots,x_n,y_n)$ consider the
symplectic form ${\rm d}x_1\wedge {\rm d}y_1+\ldots+{\rm d}x_n\wedge {\rm d}y_n$.
A smooth embedding $F: U \to V$ between open subsets of $\R^{2n}$ is \emph{symplectic}
if $F$ pulls back this form to itself. Let ${\rm B}^{2n}(R)$ denote the open ball of radius $R$ in
$\mathbb{R}^{2n}$, where $R>0$, that is, the set of points
$(x_1,y_1\ldots,x_n,y_n) \in \mathbb{R}^{2n}$ such that $\sum_{i=1}^n
(x_i)^2+(y_i)^2<R^2$.

\begin{question} (Hind and Kerman \cite[Question
  2]{HiKe2009}). \label{0} \emph{Can ${\rm B}^2(1) \times {\rm
    B}^{2(n-1)}(S)$ be symplectically embedded into ${\rm B}^4(R)
  \times \R^{2(n-2)}$ for arbitrarily large $S > 0$? If so, what is
  the smallest $R>0$ for which this is possible?}
\end{question}

\vspace{1mm}

Question \ref{0} was settled by Guth and Hind\--Kerman \cite{Guth2008,
  HiKe2009} for all numbers $R$ but one: $R=\sqrt{3}$. 
They proved that there are embeddings when $R > \sqrt{3}$
for all $S>0$, but there are not such embeddings if $R<\sqrt{3}$ and $S$ is sufficiently large.
Prior to their work it was known that the Ekeland\--Hofer capacity
implied $R > \sqrt{2}$, if such embeddings did exist (see
\cite{EkHo1989}). 

Let ${\rm Z}^{2n}(r)$ denote the cylinder of radius $r$ in
$\mathbb{R}^{2n}$, where $r>0$, that is, the set of points
$(x_1,y_1\ldots,x_n,y_n) \in \mathbb{R}^{2n}$ such that $
(x_1)^2+(y_1)^2<r^2$.   The goal of this paper
is to show the following theorem about symplectic embeddings of cylinders, which
in particular completes the answer to Question \ref{0} by answering the end\--point case.

\begin{theorem} \label{main} The cylinder ${\rm Z}^{2n}(1)$  embeds symplectically into the product 
${\rm B}^4(R)  \times \mathbb{R}^{2(n-2)}$ if $R\geq \sqrt{3}$. 
\end{theorem}

It follows from combining Guth \cite{Guth2008}, Hind\--Kerman \cite{HiKe2009},  and Theorem \ref{main} that
the cylinder ${\rm Z}^{2n}(1)$  embeds symplectically into the product 
${\rm B}^4(R)  \times \mathbb{R}^{2(n-2)}$ if and only if $R\geq \sqrt{3}$. 
The proof of Theorem \ref{main} relies  on  \cite{Guth2008, PeVN2012} and follows
closely  essential ideas of \cite{HiKe2009}. 
  
\begin{remark}
 Theorem \ref{main} can be used to derive an alternative proof of the inexistence of symplectic
  $d$\--capacities ($1<d<n$) proven in \cite{PeVN2012}.
\end{remark}

\section{Smoothness of families and Guth's Lemma}

Following \cite[Section 3]{PeVN2012}, 
  let $P,M,N$ be smooth manifolds and let $(B_p)_{p\in P}$ be a family of
  submanifolds of $N$.  We say that a family of embeddings 
  $(\phi_p \colon B_p \hookrightarrow M)_{p\in P}$ is a \emph{smooth
    family of embeddings} if~:
  \begin{enumerate}
  \item there is a smooth manifold $B$ and a smooth map $g:P\times B\to N$ such
    that $g_p:b\mapsto g(p,b)$ is an immersion and $B_p=g(p,B)$, for
    every $p \in P$;
  \item the map $\Phi:P\times B \to M$ defined by $\Phi(p,b) :=
    \phi_p\circ g(p,b)$ is smooth.
  \end{enumerate}
  In this case we also say that $(\phi_p \colon B_p \hookrightarrow M_p)_{p\in P}$
  is a \emph{smooth family of embeddings} when $M_p$ is a submanifold
  of $M$ containing $\phi_p(B_p)$.  If $M$ and $N$ are symplectic,
  then a \emph{smooth family of symplectic embeddings} is a smooth
  family of embeddings $(\phi_p)_{p\in P}$ such that each
  $\phi_p:B_p\hookrightarrow M$ is symplectic.   If $P$ is a subset of a smooth
  manifold $\tilde{P}$,  the family $(\phi_p)_{p\in
    P}$ is \emph{smooth} if there is an open neighborhood $U$ of $P$
  such that the maps $g:P\times B\to N$ and $\Phi:P\times B \to M$ may
  be smoothly extended to $U \times B$.
 
For the proof of Theorem~\ref{main} we will use the following.

\begin{theorem}[\cite{PeVN2012}] \label{cor} Let $N$ be a symplectic manifold, and let
  $W_t\subset N$, $t \in (0,\,a)$, be a family of simply connected open subsets with
  $\overline{W_s} \subset W_t$, for $s,t \in (0,\,a)$ and $t<s$.  Let
  $W_0:=\bigcup_{t \in (0,\,a)} W_t.$ Let $(\phi_t \colon W_t \hookrightarrow
  M)_{t \in (0,\,a)}$ be a smooth family of symplectic embeddings such
  that for any $t,s>0$, the set $\bigcup_{v\in [t,s]}\phi_v(W_v)$ is
  relatively compact in $M$.  Then there is a symplectic embedding
  $W_0 \hookrightarrow M$.
\end{theorem}

We will also use the following result, which is a smooth family version
of a result of Larry Guth \cite[Section 2]{Guth2008}.  As before, $n\geq 3$.

\begin{lemma}[\cite{PeVN2012}] \label{pp:10} 
 Let
$\Sigma$ be the symplectic torus
$\mathbb{T}^2=\mathbb{R}^2/\mathbb{Z}^2$ of area $1$ minus the
``origin'' (\emph{i.e.} minus the lattice $\Z^2$, $\Sigma=(\R^2 \setminus \Z^2)/\Z^2$). There is a smooth
  family $(i_{R})_{R>1/3}$ of symplectic embeddings $i_{R} \colon {\rm
    B}^{2(n-1)}(R) \hookrightarrow \Sigma \times {\rm
    B}^{2(n-2)}(10R^2).$
\end{lemma}

\section{\textcolor{black}{The Simple
    Spiral}} \label{sec:simplespiral}

The following lemma is similar to several statements in Schlenk's book
\cite{Schlenk2005}.  As in the previous section, we use the following notation: $ {\rm
  R}(A,B):=(0,A) \times (0,B)$ and ${\rm Q}(A):={\rm R}(A,A)$
($A,B>0$).

\begin{figure}[h]
  \centering \label{spiral}
  \includegraphics[width=0.4\textwidth]{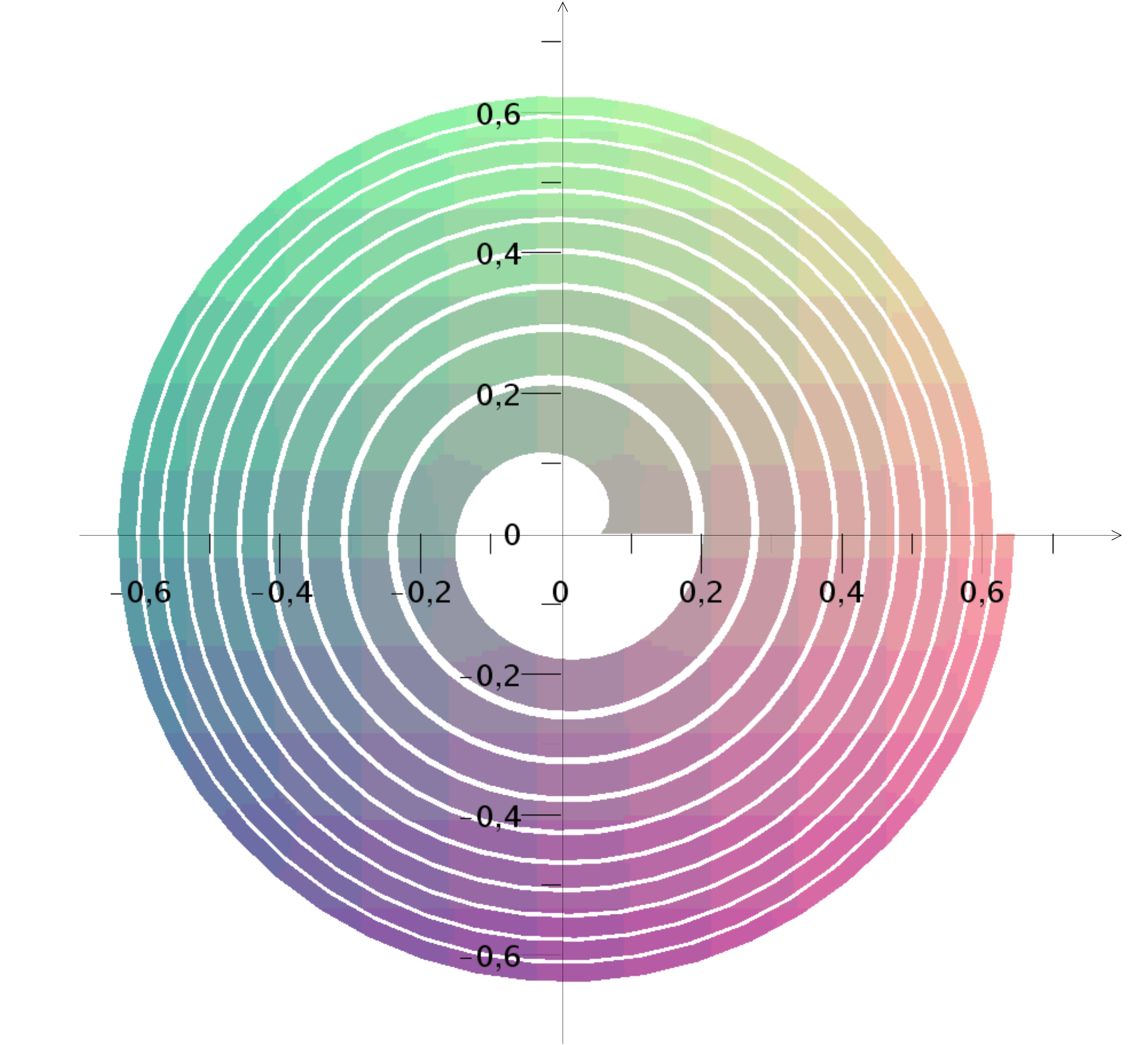}
  \caption{Numerical simulation of the simple spiral
    $\varphi_{A,B,\lambda,\delta,r} \colon (x,y) \mapsto (u,v)$ in
    Lemma \ref{mainlemma}.}
  \label{fig:sr}
\end{figure}

\begin{lemma}[Simple Spiral Lemma] \label{mainlemma} For any values
  $A>0,\,\,B>0,\,\, \lambda>0,\,\,\delta\geq 0,\,\, r\geq 0,$ the map
  \begin{eqnarray}
    \varphi_{A,B,\lambda,\delta,r}  \colon {\rm R}(A,B) \to \mathbb{R}^2,\,\,\,\,\,\,\, \,\,
    (x,y) \mapsto (u,v) \label{uv}
  \end{eqnarray}
  given by the formulas
  \begin{eqnarray}
    \begin{cases}
      u=\sqrt{\frac{I}{\pi}} \cos (2\pi \theta) \\
      v=\sqrt{\frac{I}{\pi}} \sin (2\pi \theta), \label{polar}
    \end{cases}
  \end{eqnarray}
  where $I$ and $\theta$ are given by $I=y\lambda +r
  +\frac{x}{\lambda}(B \lambda + \delta)$ and
  $\theta=\frac{x}{\lambda} \mod 1$, satisfies the following
  properties:
  \begin{enumerate}[{\rm 1.}]
  \item $\varphi_{A,B,\lambda,\delta,r}$ is a symplectic
    embedding of ${\rm R}(A,B)$ into ${\rm B}^2(r)$, where the radius
    $r_A$ is given by $
    r_A=\sqrt{\frac{B\lambda+r+AB+\frac{A\delta}{\lambda}}{\pi}}.  $
  \item $\varphi_{A,B,\lambda,\delta,r}$ maps any
    subrectangle $ {\rm R}(L,B)=(0,L) \times (0,B)$ for all 
    $L\leq A, $ of ${\rm R}(A,B)$ into ${\rm B}^2(r_L)$, where
    \begin{eqnarray} \label{radius}
      r_L=\sqrt{\frac{B\lambda+r+LB+\frac{L\delta}{\lambda}}{\pi}}.
    \end{eqnarray}
    (See Figure \ref{spiral}).
  \item The image of $\varphi_{A,B,\lambda,\delta,r}$ avoids the
    closed ball $\overline{\textup{B}^2(\sqrt{r/\pi})}$.
  \item Let $P$ be the closed subset of $\mathbb{R}^5
    \times \mathbb{R}^5 \times \mathbb{R}^5$~:
$$
P=(\mathbb{R}^*_{+})^5 \, \cup \Big(
\underbrace{(\mathbb{R}^*_{+})^4}_{A,B,\lambda,\delta} \times
\{r=0\} \Big) \, \cup \Big(
\underbrace{(\mathbb{R}^*_{+})^4}_{A,B,\lambda,r} \times
\{\delta=0\} \Big),
$$ 
where $\mathbb{R}^*_{+}$ denotes the set of strictly positive real
numbers. Then the family
$(\varphi_{A,B,\lambda,\delta,r})_{(A,B,\lambda,\delta,r) \in P}$ is
smooth.
\end{enumerate}
\end{lemma}

\begin{proof}

  Consider the symplectic maps:
  \begin{eqnarray}
    &&\left(
      \begin{array}{cc}
        \frac{1}{\lambda} & 0\\ 0 & \lambda
      \end{array}
    \right)  \colon \mathbb{R}^2 \to \mathbb{R}^2,\label{map1} \\
    &&(\cdot,\cdot) +  (0,r)  \colon \mathbb{R}^2 \to \mathbb{R}^2,\label{map2} \\
    &&\left(
      \begin{array}{cc}
        1 & 0\\ B\lambda +\delta & 1
      \end{array}
    \right)  \colon \mathbb{R}^2 \to \mathbb{R}^2, \label{map3} \\
    &&\textup{projection onto} \,\, \mathbb{R} \times (\mathbb{R}/\mathbb{Z}) \colon
    \mathbb{R}^2 \to \mathbb{R} \times (\mathbb{R}/\mathbb{Z}). \label{map4}
  \end{eqnarray}
  The composition of these maps gives the symplectomorphism depicted
  in Figure \ref{process}, and is expressed by the following formulas:
  \begin{eqnarray} \label{mm} \,\,\,\,\,\,\,\,\,\,\,(x,y)
    \stackrel{(\ref{map1})} \longmapsto
    \Big(\frac{x}{\lambda},y\lambda \Big) \stackrel{(\ref{map2})}
    \longmapsto \Big(\frac{x}{\lambda},y\lambda+r\Big)
    \stackrel{(\ref{map3})}\longmapsto \Big(\frac{x}{\lambda},
    y\lambda+r+\frac{x}{\lambda}(B\lambda+\delta) \Big),
  \end{eqnarray}
  where each function in (\ref{mm}) is restricted to its domain in
  Figure \ref{process} .
  \begin{figure}[h]
    \centering \label{process}
    \includegraphics[width=1\textwidth]{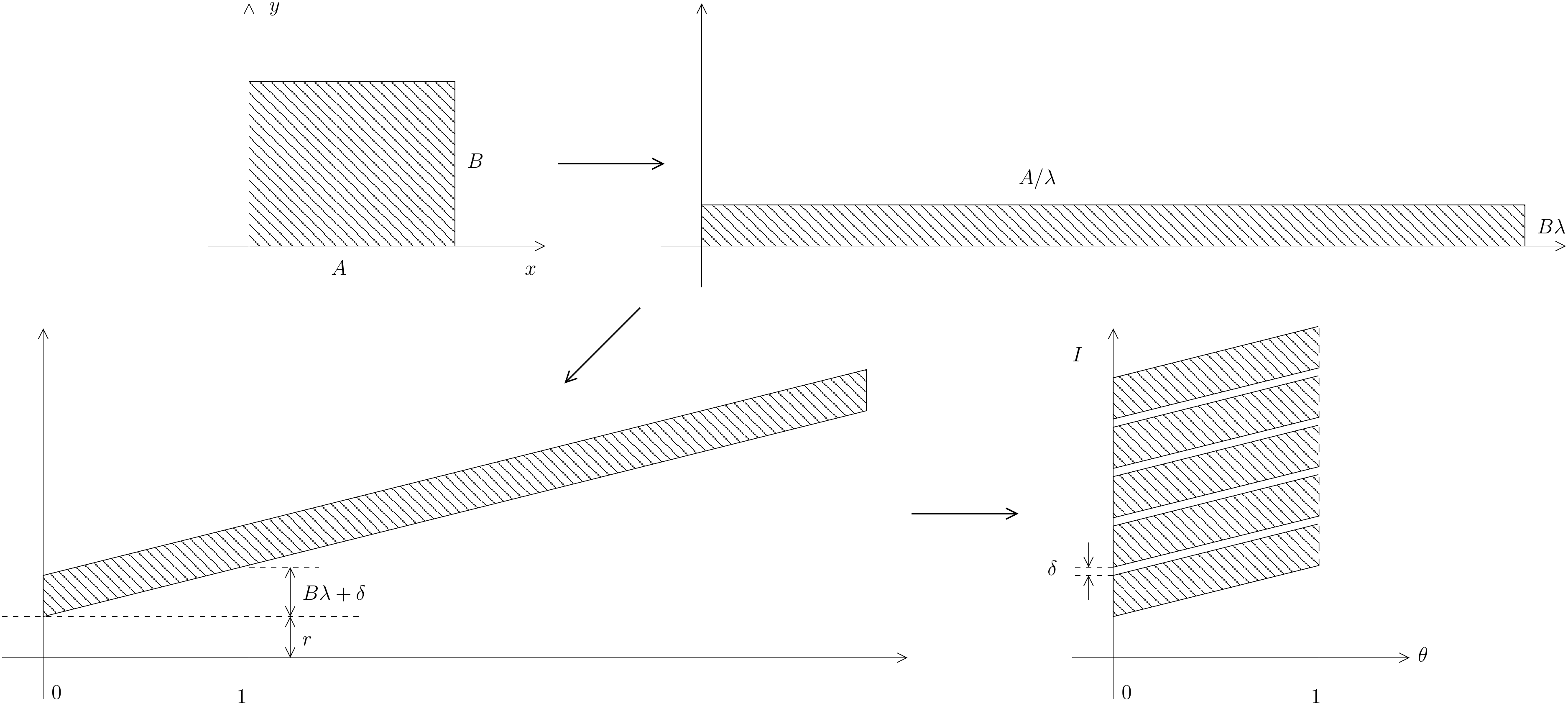}
    \caption{Domains and images of the transformations given by
      formulas (\ref{map1}), (\ref{map2}), (\ref{map3}),
      (\ref{map4}).}
    \label{process}
  \end{figure}
  Then we compose the map (\ref{mm}) with symplectic polar coordinates
  $\mathbb{R}^* \times (\mathbb{R}/\mathbb{Z}) \to \mathbb{R}^2
  \setminus \{(0,0)\}$, away from the singularity as in (\ref{polar}),
  and in this way obtain a symplectic embedding
  $\varphi_{A,B,\lambda,\delta,r}$ given in the statement of the
  lemma.  The fact that $ \varphi_{A,B,\lambda,\delta,r}$ is injective
  follows from $\delta\geq 0$ -- see the Figure \ref{process} -- and
  the slope of the line in the third part of the figure is $B\lambda
  +\delta$. This can be also be easily checked from the formulas for
  $u$ and $v$.  Finally, smoothness of the family follows from the
  fact that all transformations depend smoothly on the parameters in
  $P$. The singularity of the polar coordinates (\ref{polar}) is not
  included in the domain because the rectangles are open.
\end{proof}

\section{Review of \cite[Section~6]{PeVN2012}}

We need to use in the following sections the construction of a symplectic embedding
given in \cite[Section~6]{PeVN2012}, and because it is essential for the proof,
we review it next. It was proven therein that for
  sufficiently small fixed $\epsilon>0$ one can construct a  symplectic
  immersion $i_{\epsilon} \colon
    \Sigma(\tilde{\epsilon}) \hookrightarrow \mathbb{R}^2$
  where
  $\tilde{\epsilon}:=100\epsilon$  as in Figure \ref{fig:fancyimmersion}, with $a=\epsilon^2$.  In particular,
  the double points of the immersion are concentrated in the small
  region $[-a,a]\times[-\epsilon/2,\epsilon/2]$. 
    \begin{figure}[h]
    \centering
    \includegraphics[width=0.85\textwidth]{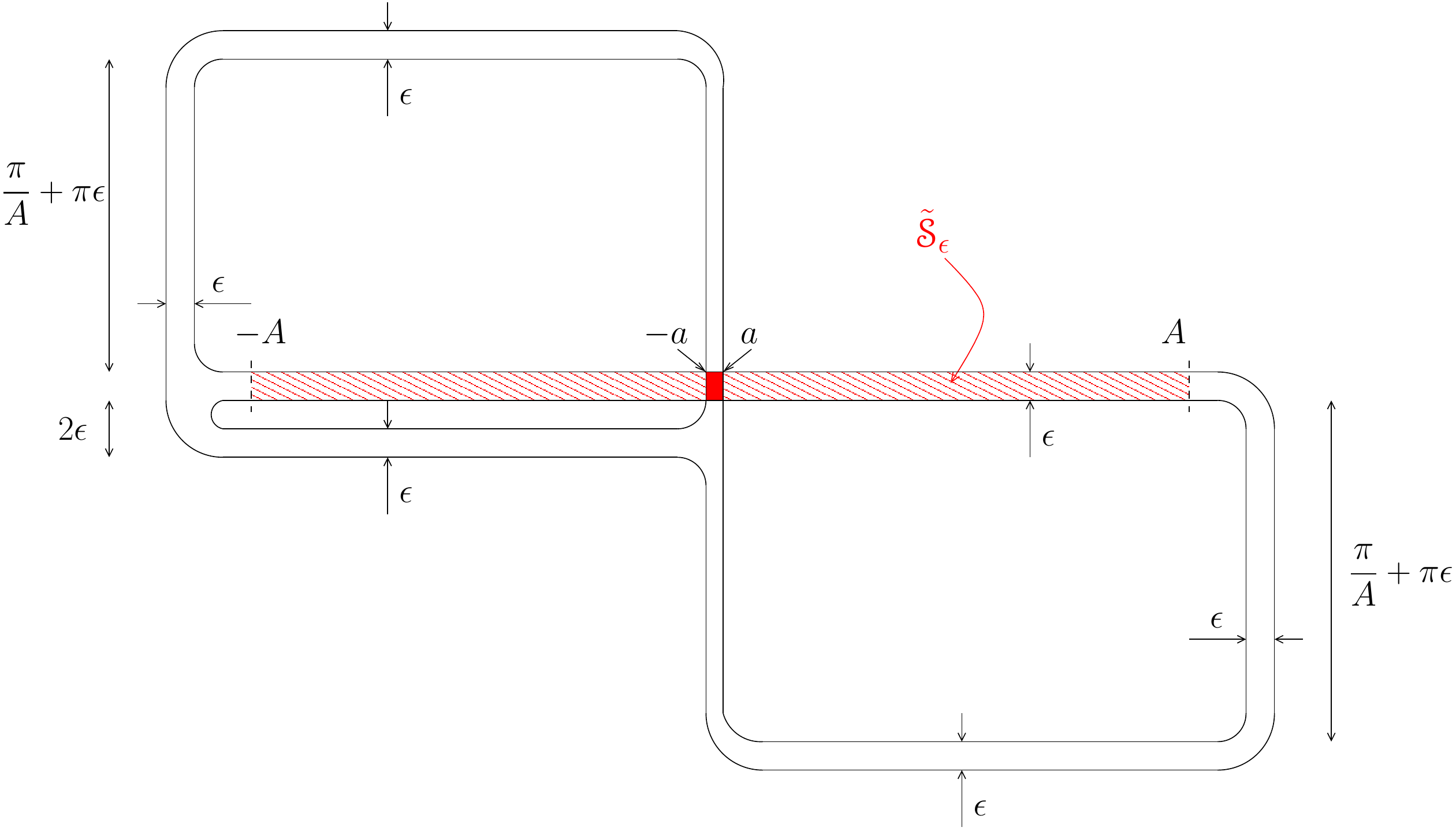}
    \caption{The immersion $i_{\epsilon} \colon
      \Sigma(\tilde{\epsilon}) \hookrightarrow \mathbb{R}^2$.}
    \label{fig:fancyimmersion}
  \end{figure}
  Consider a  smooth cut\--off function $\chi_{\epsilon} \colon \mathbb{R} \to
  [0,1]$ which is non decreasing on $\R^-$, non increasing on $\R^+$,
  $ \chi_{\epsilon} \equiv  1$ on $[-a,\,a]$,
    $\chi_{\epsilon} \equiv  0 $ on $\mathbb{R} \setminus [-A+\epsilon^2,\,
    A-\epsilon^2]$,  such that $|\chi'_{\epsilon}(x)| \leq
      \frac{1}{A}+ \epsilon$  for every $x \in \mathbb{R}$,  and  for every $x \in
    [-A+\frac{\epsilon}{2},\,A-\frac{\epsilon}{2}]$, one has that
    \begin{eqnarray}
      \biggl\rvert \chi_{\epsilon}(x) - \Big(1-\frac{|x|}{A}\Big) \biggr\rvert  \leq \epsilon. \label{trickyinequality}
    \end{eqnarray}
 On
  $\mathbb{R}^2 \times \mathbb{R}^2$ we define the smooth family of
  Hamiltonian functions $$(\mathcal{H}_{\epsilon}(x_1,y_1,x_2,y_2):=-
  \chi_{\epsilon}(x_1) x_2 \, \sqrt{\pi})_{\epsilon}$$ whose time\--$1$ flows are given by
  the smooth family $(\Phi_{\epsilon})_{\epsilon}$~:
  \begin{eqnarray} \label{equ:flow}  \nonumber \,\,\,\,\,\,\,\,\,\,
    \Phi_{\epsilon}(x_1,y_1,x_2,y_2)=\Big(x_1,\,\,y_1+\chi'_{\epsilon}(x_1)x_2\sqrt{\pi},\,\,x_2,\,\,y_2+
    \chi_{\epsilon}(x_1)\sqrt{\pi} \Big).
  \end{eqnarray}
  Let ${\rm Q}(\sqrt{\pi})$ denotes the open square $ (0,\sqrt{\pi})
  \times (0,\sqrt{\pi}) $ and ${\rm R}(\sqrt{\pi},2\sqrt{\pi})$ be the
  open rectangle $(0,\sqrt{\pi}) \times (0,2\sqrt{\pi})$.  Let
  $\mathcal{S}_{\epsilon}$ be the connected subset of
  $\Sigma(\tilde{\epsilon})$ that is mapped to the horizontal strip $
  \widetilde{\mathcal{S}_{\epsilon}}=(-A,A) \times
  (-\frac{\epsilon}{2},\frac{\epsilon}{2}) $ by the immersion
  $i_{\epsilon}$ (See Figure \ref{fig:fancyimmersion}).  We define
  $\mathcal{I}_{\epsilon} \colon \Sigma(\tilde{\epsilon}) \times {\rm
    Q}(\sqrt{\pi}) \to \mathbb{R}^4 $ by
  \begin{equation}
  \mathcal{I}_{\epsilon}(\sigma,\,b):=
  \begin{cases}
    \Phi_{\epsilon}(i_\epsilon(\sigma),b)  & \text{ if } \sigma \in \mathcal{S}_{\epsilon};\\
    (i_{\epsilon}(\sigma),b)   & \text{ if } \sigma \notin \mathcal{S}_{\epsilon},\\
  \end{cases}
\label{equ:map-I}
\end{equation}
which as shown in \cite{PeVN2012} is a symplectic embedding onto 
$\mathbb{R}^2 \times {\rm R}(\sqrt{\pi},\, 2 \sqrt{\pi})$.

\section{\textcolor{black}{Embeddings into ${\rm B}^4(R) \times
    \mathbb{R}^{2(n-2)}$}} \label{sec3}

The following is a smooth family version of the main statement in Hind and Kerman \cite[Section~4.2]{HiKe2009}.

\begin{figure}[h]
  \centering
  \includegraphics[width=0.33\textwidth]{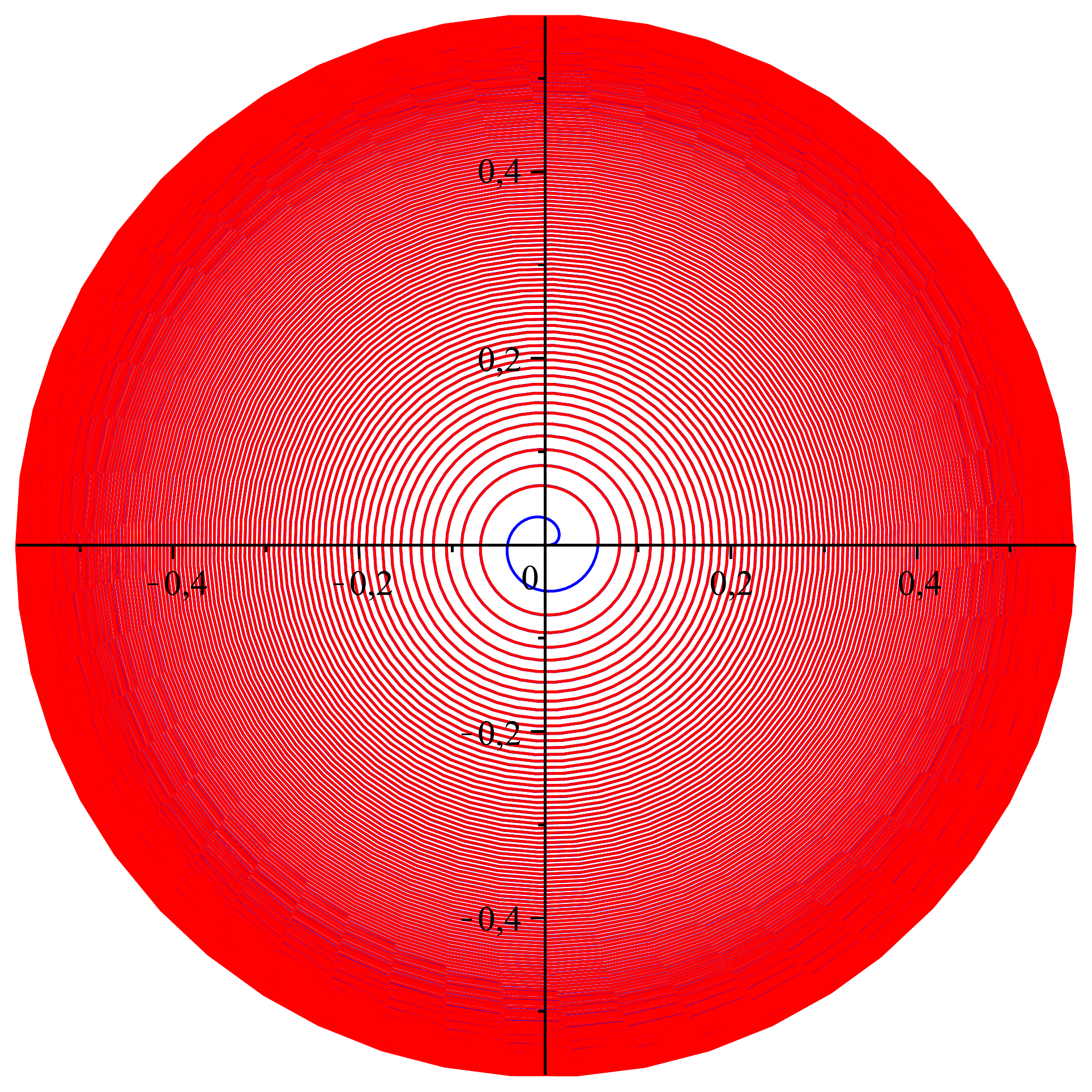}
  \caption{Numerical simulation of the  embedding in Theorem \ref{pp:12} of a square of side $1$ into
  a ball of radius $1/\sqrt{\pi}$. 
}
\label{fig:two}
\end{figure}

\begin{theorem}  \label{pp:12}
  For any $\epsilon>0$, we let
  $\Sigma(\epsilon):=(\R^2\setminus\sqrt\epsilon\Z^2)/\sqrt\epsilon\Z^2$
  be the scaling of $\Sigma$ with symplectic area $\epsilon$.  There
  exist constants $\epsilon_0>0$, $c>0$, and a smooth family
  $({J}_{\epsilon})_{\epsilon \in (0,\epsilon_0]}$ of symplectic
  embeddings ${J}_{\epsilon} \colon \Sigma(\epsilon) \times {\rm
    B}^2(1) \hookrightarrow {\rm B}^4(\sqrt 3 + c\epsilon).  $
\end{theorem}

\begin{proof} We will construct the embeddings
  explicitly, using spiral constructions.
  \\
\\
  \emph{Step 1} (\emph{A new embedding for
    $R(\sqrt\pi,2\sqrt\pi)$}). We define $F:R(\sqrt\pi,2\sqrt\pi) \to \R^2$ to be the vertical
  analogue of the simple spiral
  \begin{eqnarray} \label{parameters}
    \varphi_{A,B,\lambda,r,\delta},\,\,
    \textup{with}\,\,A=2\sqrt{\pi},\,\,\, B=\sqrt{\pi},\,\,\,
    \lambda=\epsilon, \,\,\, r=0, \,\,\,\delta=0
  \end{eqnarray}
  in Lemma \ref{mainlemma} (and Figure \ref{fig:sr}). Precisely,
  we define
\[
F:= \varphi_{A,B,\lambda,r,\delta} \circ R,
\]
where $R$ is the rotation of angle $-\pi/2$ around the origin,
followed by the translation of vector $(0,\sqrt\pi)$.
\\
\\
\emph{Step 2} (\emph{A new embedding of $D_\epsilon$}).  The
construction of a new embedding $\Phi_\epsilon$ for $D_\epsilon$ is a
bit more involved. The domain $D_\epsilon$ can be covered by two
rectangles $R_1,R_2$ (see Figure~\ref{fig:rectangles}), and each rectangle will
be sent to a spiral, in such a way that the spirals don't overlap each
other, and that there is enough space left in the image to properly
glue the two spirals together.
  \begin{figure}[h]
    \centering
    \includegraphics[width=0.7\textwidth]{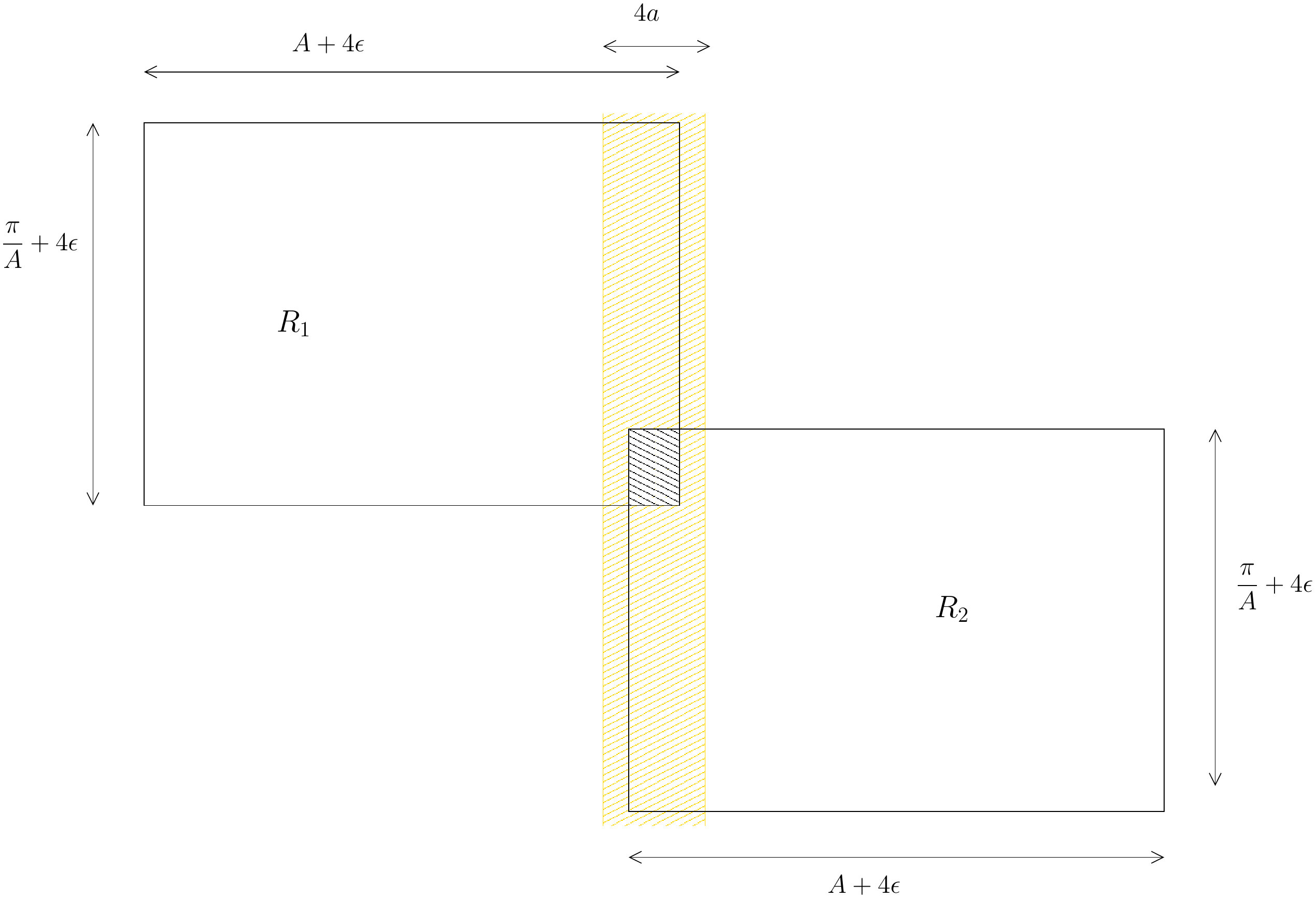}
    \caption{Rectangles $R_1,\,R_2$, and the central strip $E(4a)$.}
    \label{fig:rectangles}
  \end{figure}

  We identify $R_2$ with the rectangle
  $R(A+4\epsilon,\tfrac{\pi}{A}+2\epsilon)$ and spiral it
  with the symplectic embedding
  $\varphi_{\tilde{A},B,\lambda,r,\delta}$, given by (\ref{uv}), where
  the parameters are~:
  \begin{eqnarray}\,\,\,\,\,\,\,\,\,\,\, \tilde{A}=A+4\epsilon,
    \,\,\,\,\,\,\, B=\frac{\pi}{A}+4\epsilon, \,\,\,\,\,\,\,
    \lambda=\frac{\epsilon}{B}, \,\,\,\,\,\,\, r=M\epsilon, \,\,\,
    \,\,\,\, \delta=\epsilon,
    \label{fundamental}
  \end{eqnarray}
  where the constant $M>0$ will be determined later.  Thus we have a
  symplectic embedding $\beta_{2} \colon R_2 \to \mathbb{R}^2$.
  Similarly, we have a symplectic embedding $\beta_{1} \colon R_1 \to
  \mathbb{R}^2$ by rotating $R_1$ by the angle $\pi$ and translating
  it so that its lower right corner is at the origin $(0,0)$; we
  obtain $R(\tilde A,B)$ and then we spiral it with a modified
  symplectic embedding $\tilde\varphi_{\tilde{A},
    B,\lambda,r,\delta}$, which is given as in Lemma~\ref{mainlemma},
  except that instead of $\theta=\frac{x}{\lambda}$ we use
  $\theta=\frac{x}{\lambda}+\frac{1}{2}$. See Figure~\ref{fig:doublespiral}.
  \begin{figure}[h]
    \centering
    \includegraphics[width=0.3\textwidth]{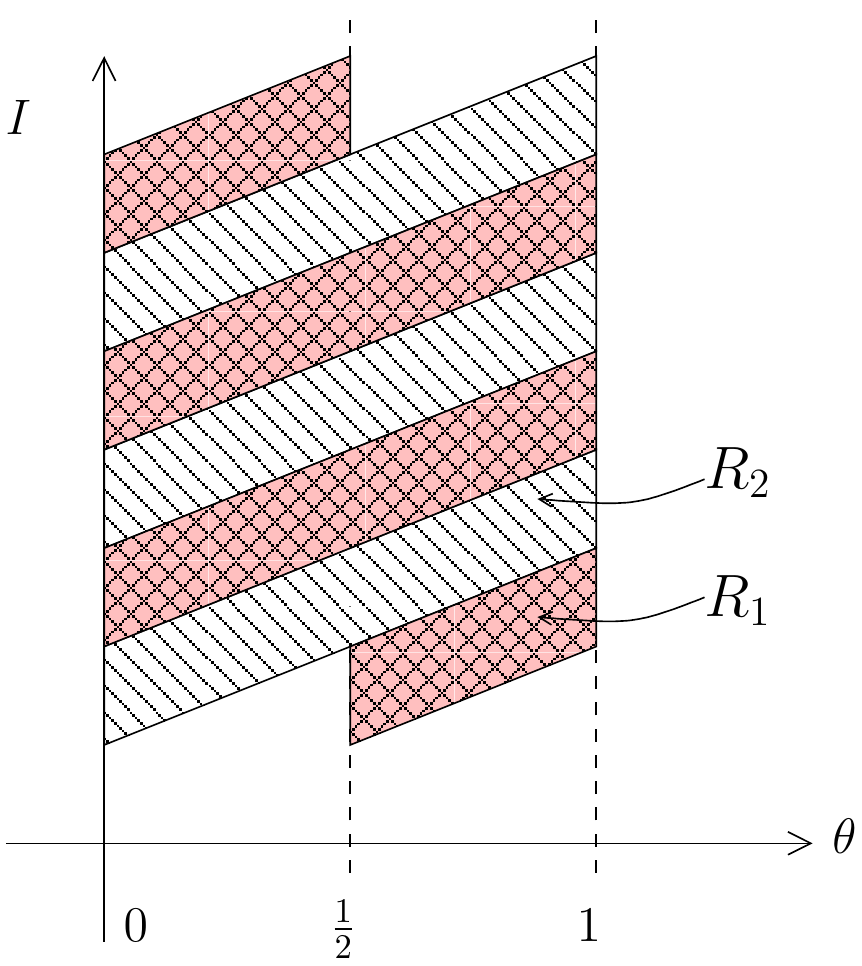}
    \caption{Construction of the double spiral embedding}
    \label{fig:doublespiral}
  \end{figure}

  For $b>0$ we denote by $E(b)$ be the vertical strip
  $(-b/2,b/2)\times\R$.  The final embedding $\Phi_\epsilon$ will be
  obtained by glueing the restrictions $\beta_{1}\vert_{R_1 \setminus
    E(4\epsilon^2)}$ and $\beta_{2}\vert_{R_2 \setminus
    E(4\epsilon^2)}$ to the central piece $W:=(R_1\cup R_2)\cap
  E(4\epsilon^2)$ (see Figure~\ref{fig:gluing}). 
    \begin{figure}[h]
    \centering
    \includegraphics[width=0.35\textwidth]{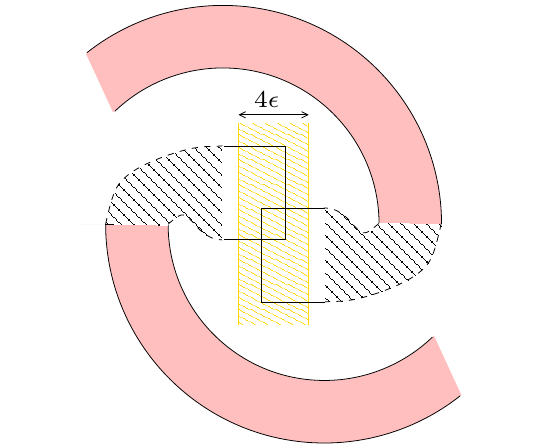}
    \caption{Gluing simple spirals in Step 1 of Theorem \ref{pp:12}.}
    \label{fig:gluing}
  \end{figure}
  This can be done by
  sending $W$ inside the ball of radius
  $\sqrt{r/\pi}=\sqrt{M\epsilon/\pi}$, which is possible for $M$ large
  enough, since the area of $W$ 
  is $\mathcal{O}(\epsilon^2)$ (Lemma \ref{mainlemma}, part 3.)
\\
\emph{Step 3} (\emph{Definition of
    $\mathcal{J}_{\epsilon}$}). Let $$\mathcal{J}_{\epsilon} \colon
  \Sigma(\epsilon) \times {\rm Q}(\sqrt{\pi}) \to \mathbb{R}^2 \times
  \mathbb{R}^2$$ be defined by $
  \mathcal{J}_{\epsilon}:=(\Phi_\epsilon\otimes F) \circ
  \mathcal{I}_{\epsilon}$, where
   $\mathcal{I}_{\epsilon} \colon \Sigma(\tilde{\epsilon}) \times {\rm
    Q}(\sqrt{\pi}) \to \mathbb{R}^4 $ 
   was defined in formula~(\ref{equ:map-I}).     We'll write $
  (x_1,y_1,x_2,y_2)=\mathcal{I}_\epsilon(\sigma,b)$ and
  $(z_1,z_2)=\mathcal{J}_{\epsilon}(\sigma,b) $ and hence
  $z_1=\Phi_\epsilon(x_1,y_1)$ and $z_2=F(x_2,y_2)$.  Our next goal is
  to show that there is some constant $c>0$, independent of
  $\epsilon,z_1,z_2$, such that
  \begin{eqnarray} \label{bb} \mathcal{J}_{\epsilon}(\Sigma(\epsilon)
    \times {\rm Q}(\sqrt{\pi}))\subset {\rm B}^4(\sqrt{3}+c \epsilon),
  \end{eqnarray}
  and in order to do this, we will find upper estimates for $|z_1|$
  and $|z_2|$.
\\
\\
  \emph{Step 4} (\emph{The image of $\mathcal{J}_{\epsilon}$}).  In
  this step we will repeatedly use the formulas in
  (\ref{fundamental}). Consider the subrectangle $\hat{R}:={\rm
    R}(y_2,A)=(0,y_2) \times (0,A).$ Using the
 \begin{figure}[h]
   \centering
   \includegraphics[width=0.4\textwidth]{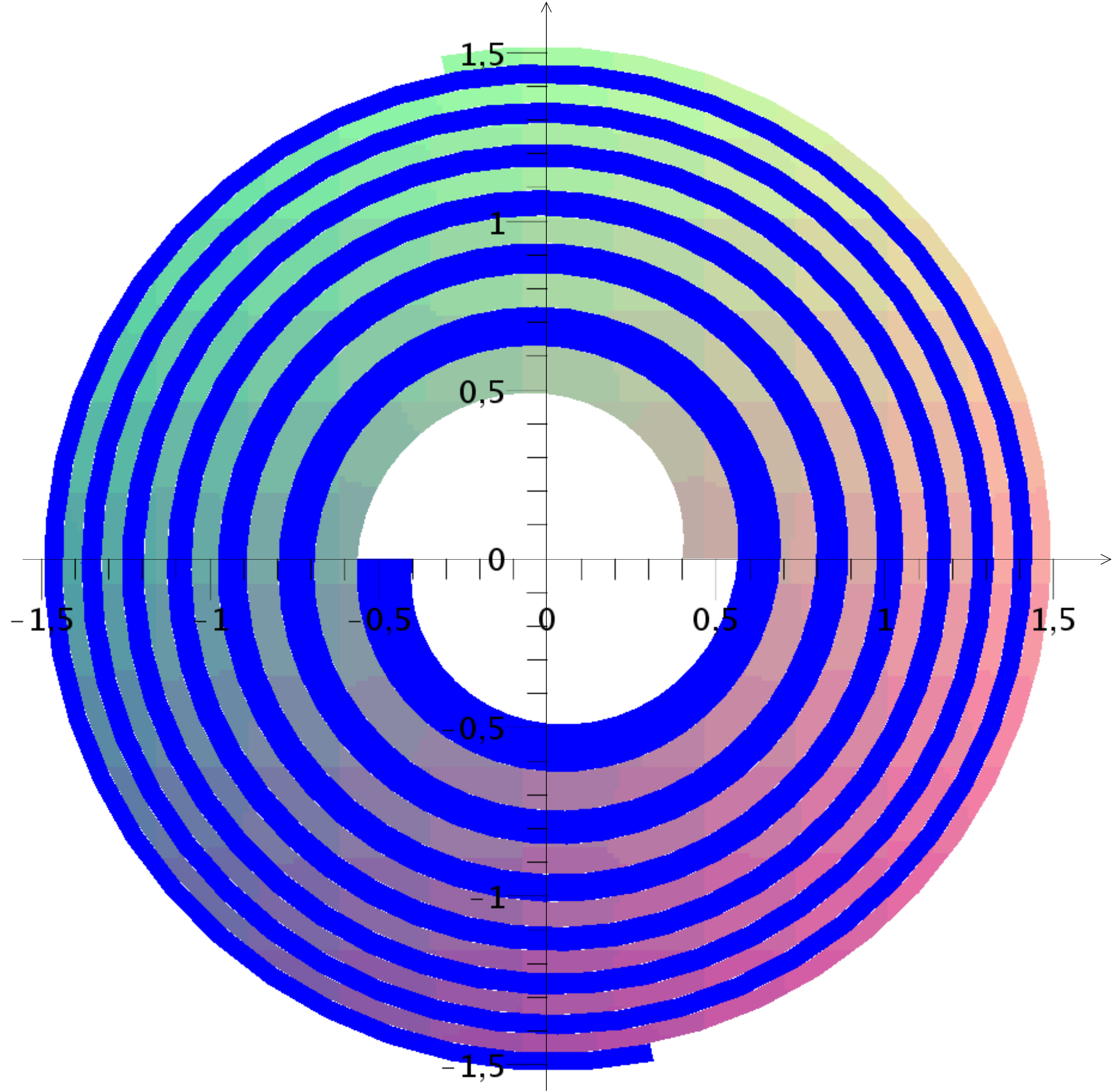}
   \caption{Numerical approximation of $\Phi_\epsilon$ in Step 2 of
     proof of Theorem \ref{pp:12}, \emph{before} the gluing in Figure
     \ref{fig:gluing}.}
   \label{two}
 \end{figure}
  formulas for the parameters in (\ref{parameters}) and formula
  (\ref{radius}) we obtain an inclusion
  \begin{eqnarray} \label{r1} F(\hat{R}) \subset {\rm
      B}^2(r_{y_2}).
  \end{eqnarray}
  where
  \begin{eqnarray} \label{r2} r_{y_2}=\sqrt{\frac{\sqrt{\pi}\epsilon
        +0+\sqrt{\pi}y_2+0}{\pi}}
    =\sqrt{\frac{y_2+\epsilon}{\sqrt\pi}}.
  \end{eqnarray}
  Since $z_2\in\overline{F(\hat R)}$, we get
  \begin{eqnarray}
    |z_2| \leq \sqrt{\frac{y_2+\epsilon}{\sqrt\pi}}. \label{estimatez2}
  \end{eqnarray}
  Now we have two cases~: (i) if $\sigma \notin
  \mathcal{S}_{\epsilon}$ then $0<y_2=b_2<\sqrt{\pi}$, and (ii) if
  $\sigma \in \mathcal{S}_{\epsilon}$ then
  $y_2=b_2+\chi_{\epsilon}(x_1)\sqrt{\pi}$.  Therefore $
  0<y_2<\sqrt{\pi} +\chi_{\epsilon}(x_1) \sqrt{\pi}$, and hence the
  estimate~(\ref{trickyinequality}) implies
  \begin{eqnarray}
    0<y_2 \leq 
    \sqrt{\pi} \Big(2-\frac{|x_1|}{A}+\epsilon\Big). \label{estimatez2second}
  \end{eqnarray}
  It follows from putting together (\ref{estimatez2}) and
  (\ref{estimatez2second}) that
  \begin{eqnarray} \label{dot} |z_2|^2 \leq 2 -\frac{|x_1|}{A} +
    \epsilon(1+1/\sqrt\pi).
  \end{eqnarray}
  This concludes the estimate for $|z_2|^2$.

  Next we find an estimate for $|z_1|^2$. Recall that $(x_1,y_1)\in
  D_\epsilon$. If $(x_1,y_1)$ belongs to the central region $W$, then
  $\abs{z_1}\leq{r/\pi}=\mathcal{O}(\epsilon)$. Otherwise, we may
  assume that $(x_1,y_1)$ lies in the rectangle $R_2$ (see
  Figure~\ref{fig:rectangles}); the case $(x_1,y_1)\in R_1$ is
  symmetrically dealt with. Let us consider the subrectangle $R(x_1,
  B)$; from (\ref{radius}) and (\ref{fundamental}) we get~:
  \begin{eqnarray}
    |z_1| &\leq& \frac1{\sqrt\pi}\sqrt{\epsilon + M\epsilon
        +x_1 \Big(\frac{\pi}{A} +4\epsilon\Big) + x_1\Bigl(\frac{\pi}{A} +
        4\epsilon\Bigr)} \nonumber \\
    &\leq&   \sqrt{\frac{2x_1}{A}  +\frac{\epsilon (1+M+8x_1)}{\pi}}. \label{mts}
  \end{eqnarray}
  It follows from (\ref{mts}) that there exists a constant $C<\infty$
  (recall that $0<x_1<A+4\epsilon$) such that
  \begin{eqnarray} \label{dot2} \frac{|z_1|^2}{2} \leq
    \frac{|x_1|}{A} + C\epsilon,
  \end{eqnarray}
and in particular
\begin{equation}
  \label{dot3}
  \frac{\abs{z_1}^2}2 \leq 1 + \tilde C\epsilon, 
\end{equation}
where $\tilde C=C+4/A$.  Hence from (\ref{dot}) and (\ref{dot2}) we
get that
  \begin{eqnarray} \label{abcd} \frac{|z_1|^2}{2} + |z_2|^2 \leq 2
    +(1+C+1/\sqrt\pi)\epsilon.
  \end{eqnarray}
  Adding (\ref{dot3}) we obtain that~:
$$
|z_1|^2+|z_2|^2 \leq 3+c\epsilon,
$$
where $c:=1+2C+1/\sqrt\pi+4/A$ is a constant independent of
$\epsilon,z_1,z_2$.  Hence we get (\ref{bb}), which concludes the
proof of the theorem.
\end{proof}

The following corresponds to \cite[Theorem 1.3]{HiKe2009} for smooth
  families.

\begin{theorem} \label{kh} Let $n\geq 3$.  There exist constants $C,C'>0$
  and a smooth family of symplectic embeddings
  \[
  i_{S,\,R} \colon {\rm B}^2(1) \times {\rm B}^{2(n-1)}(S)
  \hookrightarrow {\rm B}^4(R) \times
  \textup{B}^{2(n-2)}({\textstyle\frac{C S^2}{\sqrt{R-\sqrt 3}}}),
  \] where $(S,R)$ vary in the open set
  \begin{equation}
    \{ (S,R)\in \R^2\,\, | \qquad S>0, \quad \sqrt 3 < R < \sqrt 3 + C'S^2\}.
    \label{equ:domain2}
  \end{equation}
\end{theorem}

\begin{proof}
  The proof is essentially identical to that of \cite[Theorem~6.4]{PeVN2012}.
  Consider the embedding $i_{T} \colon {\rm B}^{2(n-1)}(T)
  \hookrightarrow \Sigma \times {\rm B}^{2(n-2)}(10T^2)$,
  $T>{\textstyle\frac{1}{3}}$ in Lemma~\ref{pp:10}. For $\epsilon>0$,
  let $\tau_{\sqrt\epsilon}:\R^{2(n-1)}\to\R^{2(n-1)}$ be the dilation
  $\tau_{\sqrt{\epsilon}}(x)=\sqrt\epsilon x$. The corresponding
  quotient map $\bar\tau_{\sqrt\epsilon}$ maps $\Sigma\times
  \R^{2(n-2)}$ to $\Sigma(\epsilon) \times \R^{2(n-2)}$, The map
  $\bar\tau_{\sqrt\epsilon}\circ i_T \circ
  (\tau_{\sqrt\epsilon})^{-1}$ is a symplectic embedding of
  $\tau_{\sqrt\epsilon}({\rm B}^{2(n-1)}(T))={\rm
    B}^{2(n-1)}(\sqrt\epsilon T)$ into
  $\bar\tau_{\sqrt\epsilon}(\Sigma \times {\rm B}^{2(n-2)}(10T^2)) =
  \Sigma(\epsilon) \times {\rm B}^{2(n-2)}(10\sqrt\epsilon T^2))$. Of
  course, as $T>\frac{1}{3}$ varies, the corresponding family of
  embeddings is smooth.  By composing with the embeddings given by
  Theorem \ref{pp:12}, we end up with a smooth family of symplectic
  embeddings~:
  \begin{gather*}
    {\rm B}^2(1) \times {\rm B}^{2(n-1)}(\sqrt\epsilon T)
    \hookrightarrow {\rm B}^4(\sqrt 3 + c\epsilon) \times
    \textup{B}^{2(n-2)}(10\sqrt\epsilon T^2),\\
    T>1/3, \quad \epsilon>0.
  \end{gather*}
  The conclusion follows by the smooth parameter change 
   $(S,R):=(\sqrt\epsilon T,\sqrt 3 + c\epsilon)$, with
  $C=10\sqrt{c}$ and $C'=9c$.
\end{proof}

\section{\textcolor{black}{Proof of Theorem \ref{main}}} \label{sec:proof}

From~\cite[Theorem~1.1]{HiKe2009} we know that if $0 < R <\sqrt{3}$
there are no symplectic embeddings of ${\rm B}^2(1) \times {\rm
  B}^{2(n-1)}(S) $ into ${\rm B}^4(R) \times \mathbb{R}^{2(n-2)}$ when
$S$ is large. Therefore, it remains to prove that ${\rm B}^2(1) \times
\R^{2(n-1)}$ symplectically embeds into ${\rm B}^4(\sqrt 3) \times
\mathbb{R}^{2(n-2)}$.

The proof is analogous to the proof of \cite[Theorem~3.3]{PeVN2012}.  By
Theorem~\ref{kh} there exist some constants $C,C'>0$ and a smooth
family of symplectic embeddings $ i_{S,\,R} \colon {\rm B}^2(1) \times
{\rm B}^{2(n-1)}(S) \hookrightarrow {\rm B}^4(R) \times
\textup{B}^{2(n-2)}({\textstyle\frac{C S^2}{\sqrt{R-\sqrt 3}}}), $
where $(S,R)$ is in the region $A$ of $(S,R)\in \R^2$ such that $S>0$
and $\sqrt 3 < R < \sqrt 3 + C'S^2$.  For all $\epsilon>0$ small
enough we may define a smooth family of symplectic embeddings 
\begin{equation}
  \phi_{\epsilon} \colon {\rm B}^2(1-\epsilon) \times {\rm
    B}^{2(n-1)}(1/\epsilon) \hookrightarrow {\rm B}^4(\sqrt{3}) \times
  \textup{B}^{2(n-2)}\Bigl(\frac{3^{-1/4}C}{\sqrt{\epsilon^5(1-\epsilon)}}\Bigr)
  \label{equ:phifamily2}
\end{equation}
by $\phi_{\epsilon}(x):=(\frac{\sqrt3}{R})i_{S,R}(\frac{R}{\sqrt 3}x)$
with $S=\frac{1}{\epsilon(1-\epsilon)}$ and
$R=\frac{\sqrt{3}}{1-\epsilon}$.  We apply Theorem \ref{cor} to the
family (\ref{equ:phifamily2}) and get a symplectic embedding ${\rm
  B}^2(1) \times \mathbb{R}^{2(n-1)} \hookrightarrow {\rm
  B}^4(\sqrt{3}) \times \mathbb{R}^{2(n-2)}.$

\vspace{3mm}

{\small \emph{Acknowledgements}.  We thank Helmut Hofer for helpful discussions. 
 AP was partly supported by NSF Grant DMS-0635607,  an NSF CAREER Award DMS-1055897,
 a J. Tinsley Oden Faculty Fellowship at the Institute for Computational Engineering and Sciences (ICES)
 in Austin, a  Leibniz Fellowship from Oberwolfach, Spain Ministry of Science Grant Sev-2011-0087.  
 VNS is partially supported by the Institut
  Universitaire de France, the Lebesgue Center (ANR Labex LEBESGUE),
  and the ANR NOSEVOL grant.  He gratefully acknowledges the
  hospitality of the IAS. }

{\small
  \noindent
  \\
  {\bf {\'A}lvaro Pelayo} \\
  School of Mathematics \\
  Institute for Advanced Study\\
  Einstein Drive\\
     Princeton, NJ 08540 USA \\
  \\
  \noindent
  Washington University,  Mathematics Department \\
  One Brookings Drive, Campus Box 1146 \\
  St Louis, MO 63130-4899, USA.\\
  {\em E\--mail}: \texttt{apelayo@math.wustl.edu} \\
   {\em E\--mail}: \texttt{apelayo@math.ias.edu} 
  \noindent
  \\
  \\
  \noindent
  {\bf San V\~u Ng\d oc} \\
  Institut Universitaire de France
  \\
  \\
  Institut de Recherches Math\'ematiques de Rennes\\
  Universit\'e de Rennes 1, Campus de Beaulieu\\ F-35042 Rennes cedex, France\\
  {\em E-mail:} \texttt{san.vu-ngoc@univ-rennes1.fr}\\


\begin{thebibliography}{10}




\bibitem{EkHo1989} I. Ekeland and H. Hofer: Symplectic topology and
  Hamiltonian dynamics. \emph{Math. Z.} {\bf 200} (1989) 355--378.



  \bibitem{Guth2008} L. Guth: Symplectic embeddings of polydisks.
  \emph{Invent. Math.}  {\bf 172} (2008) 477--489.


\bibitem{HiKe2009} R. Hind and E. Kerman: New obstructions to
  symplectic embeddings. arXiv:0906.4296v2.


\bibitem{PeVN2012} \'A. Pelayo and S.  V\~{u} Ng\d{o}c: The Hofer question on
intermediate symplectic capacities. arXiv:1210.1537v3.



\bibitem{Schlenk2005} F. Schlenk: Embedding Problems in Symplectic
  Geometry. de Gruyter Expositions in Mathematics, vol. 40. Walter de
  Gruyter, Berlin (2005).
  

\end{thebibliography}
\end{document}